\documentclass{amsart}

\usepackage[backend=biber, style=alphabetic]{biblatex} 
\DeclareFieldFormat[
		article,
		book,
		inbook,
		incollection,
		inproceedings,
		patent,
		thesis,
		unpublished
	]{title}{\emph{#1\isdot}}

\usepackage{amssymb}
\usepackage{mathrsfs}

\newcommand\newterm[1]{\textit{#1}}

\usepackage{mleftright}
\newcommand\paren[1]{(#1)}
\newcommand\Paren[1]{\mleft(#1\mright)}

\newcommand\vertbar[1]{\lvert#1\rvert}

\newcommand\Angle[1]{\mleft\langle#1\mright\rangle}
\newcommand\rounddown[1]{\lfloor #1\rfloor}

\newcommand\finset[1]{\{#1\}}
\newcommand\comp\circ
\newcommand\restr[2]{#1\rvert_{#2}}

\newcommand\Zahlen{\mathbb Z}

\newcommand\Quot{\mathbb Q}

\newcommand\Comp{\mathbb C}

\newcommand\order\vertbar
\DeclareMathOperator\Image{Im}
\newcommand\gengp\Angle

\newcommand\multgp[1]{{#1}^\times}

\newcommand\affsp[1]{\mathbb A^{#1}}
\newcommand\projsp[1]{\mathbb P^{#1}}
\newcommand\wproj[1]{\mathbb P\paren{#1}}
\newcommand\hypsurf[1]{\{#1\}}

\newcommand\difsh[2]{\Omega_{#1}^{#2}}

\newcommand\Aut[1]{\operatorname{Aut}\paren{#1}}
\newcommand\AUT[1]{\operatorname{Aut}\biggl(#1\biggr)}

\DeclareMathOperator\Exc{Exc}

\newcommand\Gm{\mathbb G_{\mathrm m}}

\DeclareMathOperator\Supp{Supp}

\newcommand\lineq\sim
\newcommand\lineqQ{\lineq_\Quot}
\newcommand\numeq\equiv

\newcommand\Idx[2]{I_{#1}\paren{#2}}
\newcommand\idx[1]{\operatorname{idx}\paren{#1}}

\newcommand\klt{\mathrm{klt}}
\newcommand\sm{\mathrm{sm}}
\newcommand\term{\mathrm{term}}

\newcommand\Phist{\Phi_{\mathrm{st}}}
\newcommand\Idxstand[1]{\Idx\klt{#1,\Phist}}

\newcommand\lcm[1]{\operatorname{lcm}\paren{#1}}

\DeclareMathOperator\trace{tr}

\DeclareMathOperator\Coeff{Coeff}

\usepackage[dvipsnames]{xcolor}

\usepackage{braket}
\newcommand\sset[2]{\set{#1|#2}}
\newcommand\Sset[2]{\Set{#1|#2}}

\usepackage{tikz}
\usetikzlibrary{cd}

\usepackage[
		colorlinks,
		linkcolor=MidnightBlue,
		citecolor=ForestGreen,
		urlcolor=Maroon
	]{hyperref}

\usepackage[nameinlink]{cleveref}

\crefname{section}{Section}{Sections}
\crefname{subsection}{Section}{Sections}

\newtheorem{thm}{Theorem}[section]
\crefname{thm}{Theorem}{Theorems}
\newtheorem{prop}[thm]{Proposition}
\crefname{prop}{Proposition}{Propositions}
\newtheorem{lem}[thm]{Lemma}
\crefname{lem}{Lemma}{Lemmas}
\newtheorem{cor}[thm]{Corollary}
\crefname{cor}{Corollary}{Corollaries}
\newtheorem{conj}[thm]{Conjecture}
\crefname{conj}{Conjecture}{Conjectures}
\newtheorem{ques}[thm]{Question}
\crefname{ques}{Question}{Questions}

\theoremstyle{definition}
\newtheorem{defi}[thm]{Definition}
\crefname{defi}{Definition}{Definitions}

\crefname{exa}{Example}{Examples}

\crefname{claim}{Claim}{Claims}
\newtheorem*{ack}{Acknowledgements}

\theoremstyle{remark}
\newtheorem{rem}[thm]{Remark}
\crefname{rem}{Remark}{Remarks}

\title{Relations between indices of Calabi--Yau varieties and pairs}
\author{Yuto Masamura}
\address{Graduate School of Mathematical Sciences, the University of Tokyo, 3-8-1 Komaba, Meguro-ku, Tokyo 153-8914, Japan}
\email{masamura@ms.u-tokyo.ac.jp}

\begin{document}

\begin{abstract}
	We show that for any smooth Calabi--Yau variety,
	its index can be realized as the index of a Kawamata log terminal (klt) Calabi--Yau pair of lower dimension with standard coefficients.
	Our approach is based on an inductive argument on the dimension
	using the Beauville--Bogomolov decomposition.
	A key step in the argument is to prove that for $n\ge3$,
	any positive integer $m$ satisfying $\varphi(m)\le 2n$
	can be realized as the index of a klt Calabi--Yau pair of dimension $n-1$.
\end{abstract}

\maketitle

\tableofcontents

\section{Introduction}

We work over $\Comp$, the field of complex numbers.

Calabi--Yau pairs play a fundamental role in the minimal model program.
This paper focuses on the indices of Calabi--Yau pairs.
The \newterm{index} of a Calabi--Yau pair $\paren{X,B}$ is defined as the smallest integer $m\ge1$ such that $m\paren{K_X+B}\lineq0$.
It is expected that, with fixed dimension of $X$ and coefficients of $B$, and under suitable assumptions on singularities, the indices of Calabi--Yau pairs $\paren{X,B}$ are bounded (cf.~\cite[Conjecture 1.1]{xu}):

\begin{conj}[Index conjecture for Calabi--Yau pairs]
	\label{index conj}
	Let $n\ge1$ be an integer and $\Phi\subseteq[0,1]\cap\Quot$ be a finite set.
	Then there exists an integer $m\ge1$,
	depending only on $n$ and $\Phi$,
	satisfying the following:
	if $\paren{X,B}$ is an lc Calabi--Yau pair of dimension $n$ with $\Coeff\paren B\subseteq\Phi$, then $m\paren{K_X+B}\lineq0$.
\end{conj}

Note that the index conjecture is also formulated for slc pairs, a broader class of pairs, as indicated in \cite[Conjecture 1.3]{xu}, \cite[Conjecture 1.5]{jiang-liu}.
The index conjecture was studied in previous works such as \cite{kawamata,morrison,prokhorov-shokurov,hacon-mckernan-xu,jiang,xu-comp,xu,jiang-liu}.
Specifically, it was proven in dimension $2$ by Prokhorov and Shokurov \cite[Corollary 1.11]{prokhorov-shokurov}, for terminal $3$-folds by Kawamata \cite[Theorem 3.2]{kawamata} and Morrison \cite[Corollary]{morrison}, for Kawamata log terminal (klt) $3$-folds by Jiang \cite[Corollary 1.7]{jiang}, for lc pairs of dimension $3$ by Xu \cite[Theorem 1.13]{xu}, and for slc pairs of dimension $3$ and non-klt (lc) pairs of dimension $4$ by Jiang and Liu \cite[Corollaries 1.6, 1.7]{jiang-liu}.
Hacon, McKernan and Xu \cite[Theorem 1.5]{hacon-mckernan-xu} showed that \cref{index conj} is equivalent to the conjecture for $\Phi$ satisfying only the descending chain condition.

Inductive steps towards the index conjecture have been established, see \cite[Theorems 1.7, 1.8, 1.11, 1.12]{xu} and \cite[Propositions 4.7, 4.8]{jiang-liu}.
In this paper we restrict our attention to the index conjecture for \textit{klt} pairs.
Xu showed that the index conjecture for klt pairs can be reduced to that for klt varieties:
 
 \begin{thm}[{\cite[Theorem 1.8]{xu}}]
 	\label{xu step}
 	Let $n\ge1$ be an integer.
	Assume that \cref{index conj} holds true for klt varieties $\paren{X,0}$
	of dimension $\le n$.
	Then \cref{index conj} holds true for klt pairs in dimension $n$.
 \end{thm}

Thus, when proving the index conjecture for klt pairs by induction on the dimension $n$,
it is crucial to consider the following question:

\begin{ques}
	\label{vague q}
	Is there a relationship between the indices of klt Calabi--Yau varieties
	of dimension $n$ and those of klt Calabi--Yau pairs of dimension $<n$?
\end{ques}

Here, we introduce notations regarding the indices of Calabi--Yau pairs.
For $n\ge1$ and a subset $\Phi\subseteq[0,1]\cap\Quot$, we define $\Idx\klt{n,\Phi}$
as the set of the indices of klt Calabi--Yau pairs $\paren{X,B}$ of dimension $n$ with $\Coeff\paren B\subseteq\Phi$.
Now \cref{vague q} can be reformulated as follows:
Can we find a relationship between the set $\Idx\klt {n,\emptyset}$
and (a subset of) the set $\Idx\klt{n-1,\Phi}$ for some $\Phi$?

Moreover, 
we define $\Idx\sm n$ and $\Idx\term n$ as the sets of the indices
of $n$-dimensional smooth Calabi--Yau varieties and terminal Calabi--Yau varieties,
respectively.
We also set $\Idx\sm0=\Idx\term0=\Idx\klt{0,\Phi}=\finset1$ for any $\Phi$.
Define $\Phist$ to be the set of \newterm{standard coefficients},
which are rational numbers of the form $1-1/b$ for some $b\in\Zahlen_{\ge1}$.

There is limited knowledge on the indices of klt Calabi--Yau varieties.
However, we propose the following conjecture on the indices of terminal Calabi--Yau varieties:

\begin{conj}
	\label{index set conj}
	Let $n\ge1$ be an integer.
	Then it holds that
	\[
		\Idxstand{n-1}=\Idx\term n.
	\]
\end{conj}

Esser, Totaro and Wang studied the largest values of indices of Calabi--Yau pairs.
They expect that the largest values in both $\Idxstand{n-1}$ and $\Idx\term n$ are equal to $\paren{s_{n-1}-1}\paren{2s_{n-1}-3}$, where the sequence $s_n$ is defined by $s_0=2$ and $s_n=s_{n-1}\paren{s_{n-1}-1}+1$, see \cite[Conjectures 3.4, 4.2]{esser-totaro-wang-cal}.
\Cref{index set conj} is in accordance with their conjectures.

The main theorem of this paper is the following:

\begin{thm}[\cref{klt st subseteq term,index set in low dimension,index set thm}]
	\label{intro main thm}
	Let $n\ge1$ be an integer.
	Then it holds that
	\[
		\Idx\sm n\subseteq\Idxstand{n-1}\subseteq\Idx\term n.
	\]
	Moreover, $\Idx\sm n=\Idxstand{n-1}=\Idx\term n$ for $n\le3$.
\end{thm}

The second inclusion in \cref{intro main thm} shows that \cref{index set conj} is equivalent to the inclusion $\Idxstand{n-1}\supseteq\Idx\term n$.
The first inclusion, which is the core of the theorem, shows that the conjectural inclusion holds partially.

Since the set $\Idxstand3$ is bounded by \cite[Theorem 1.13]{xu} (and \cite[Theorem 1.5]{hacon-mckernan-xu}), we get the following corollary:

\begin{cor}
	\label{sm cy 4folds}
	The set $\Idx\sm4$ is bounded, that is, the indices of smooth Calabi--Yau varieties of dimension $4$ are bounded.
\end{cor}

\begin{rem}
	Beauville found that \cref{sm cy 4folds} is true if the indices of automorphisms (of finite order) of strict Calabi--Yau $3$-folds are bounded, see \cite[Remarks following Proposition 8]{beauville}.
	We can prove \cref{sm cy 4folds} just by using his arguments and \cite[Theorem 1.13]{xu}, along with \cref{quot and index}.
\end{rem}

We give a sketch of the proof for the first inclusion in \cref{intro main thm}, using induction on $n$.
Let $X$ be a smooth Calabi--Yau variety of dimension $n$ with index $m$, and $f\colon\tilde X\to X$ be a Beauville--Bogomolov decomposition (see \cref{decomposition}).
Then $\tilde X$ is a product of strict Calabi--Yau varieties, holomorphic symplectic varieties, and an abelian variety.
If $\tilde X$ has at least $2$ components, then the induction hypothesis implies that $m\in\Idxstand{n-1}$.
If $\tilde X$ itself is either strict Calabi--Yau, holomorphic symplectic, or abelian, then we can see that $m$ satisfies $\varphi\paren m\le 2n$ (see \cref{free aut phi<2n}).
Here, $\varphi$ is the Euler function.
We conclude the proof by using the following second main theorem of this paper:

\begin{thm}[\cref{phi<2n and index}]
	\label{intro small index}
	Let $n\ge3$ be an integer.
	Then it holds that
	\[
		\sset{m\in\Zahlen_{\ge1}}{\varphi(m)\le2n}
		\subseteq\Idxstand{n-1}.
	\]
\end{thm}

\begin{rem}
	The set $\sset{m\ge1}{\varphi(m)\le2n}$ is much smaller than $\Idxstand{n-1}$.
	Specifically, for $n=3$, we have
	$\Idxstand{2}=\sset{m\ge1}{\varphi(m)\le20}\setminus\finset{60}$
	by \cref{index set in low dimension}.
\end{rem}

\begin{ack}
	The author would like to thank his advisor, Professor Keiji Oguiso, for providing support and encouragement.
	The author is also grateful to Professor Chen Jiang for the helpful feedback on an earlier draft, and to Professor Burt Totaro for the invaluable guidance and detailed responses to inquiries.
	The author appreciates the referees for reading the manuscript carefully and offering useful suggestions.
\end{ack}

\section{Preliminaries}

\subsection{Pairs and singularities}

Throughout this paper, a \newterm{variety} is assumed to be irreducible and reduced.
A \newterm{divisor} (resp.\;\newterm{$\Quot$-divisor}) on a normal variety is a finite sum $\sum_id_iD_i$ of prime divisors $D_i$
with integral (resp.\;rational) coefficients $d_i$.
For two $\Quot$-divisors $D,D'$ on a normal variety, we write $D\lineq D'$ (resp.\;$D\lineqQ D'$) for linear (resp.\;$\Quot$-linear) equivalence.
A $\Quot$-divisor $D=\sum_id_iD_i$ is \newterm{effective}, denoted by $D\ge0$, if $d_i\ge0$ for each $i$.
We define $\Coeff(D)=\sset{d_i}{d_i\ne0}$ to be the set
of nonzero coefficients of a $\Quot$-divisor $D=\sum_id_iD_i$.
Define $\Phist=\sset{1-1/b}{b\in\Zahlen_{\ge1}}$ to be the set of \newterm{standard coefficients}.

A \newterm{pair} $\paren{X,B}$ consists of a normal variety $X$ and a $\Quot$-divisor $B\ge0$ such that $K_X+B$ is $\Quot$-Cartier.

A pair $\paren{X,B}$ is \newterm{Kawamata log terminal} (\newterm{klt} for short)
if for any log resolution $f\colon Y\to X$ with $K_Y+B_Y=f^*\paren{K_X+B}$,
the $\Quot$-divisor $B_Y$ has coefficients $<1$.
A normal variety $X$ is \newterm{terminal} if $K_X$ is $\Quot$-Cartier
and if for any resolution $f\colon Y\to X$ with $K_Y+B_Y=f^*K_X$,
the $\Quot$-divisor $B_Y$ has coefficients $<0$ and $\Supp B_Y=\Exc\paren f$.

\subsection{Calabi--Yau pairs}

A \newterm{Calabi--Yau pair} is defined as a projective pair $\paren{X,B}$ satisfying $K_X+B\lineqQ0$.
The \newterm{index} of a Calabi--Yau pair $\paren{X,B}$ is defined as
\[
	\idx{X,B}=\min\sset{m\in\Zahlen_{\ge1}}{m\paren{K_X+B}\lineq0}.
\]
If $B=0$, we say that $X$ is a \newterm{Calabi--Yau variety}, and write $\idx X=\idx{X,0}$.

For $n\ge1$ and $\Phi\subseteq[0,1]\cap\Quot$, we set
\begin{align*}
	\Idx\sm n
	&=\sset{\idx X}
		{\text{$X$ is a smooth Calabi--Yau variety of dimension $n$}},\\
	\Idx\term n
	&=\sset{\idx X}
		{\text{$X$ is a terminal Calabi--Yau variety of dimension $n$}},\\
	\Idx\klt{n,\Phi}
	&=\Sset{\idx{X,B}}
		{\genfrac{}{}{0pt}{}{\text{$\paren{X,B}$ is a klt Calabi--Yau pair}}
			{\text{of dimension $n$ with $\Coeff(B)\subseteq\Phi$}}}.
\end{align*}
We also set $\Idx\sm0=\Idx\term0=\Idx\klt{0,\Phi}=\finset1$.

\begin{lem}
	Let  $n\ge1$ be an integer and $\Phi\subseteq[0,1]\cap\Quot$ be a subset.
	Then it holds that
	\[
		\Idx\sm n\subseteq\Idx\term n\subseteq\Idx\klt{n,\Phi}.
	\]
\end{lem}

\begin{lem}
	\label{properties of Idx}
	Let $\Phi\subseteq[0,1]\cap\Quot$.
	For $n\ge1$, let $I(n)$ denote one of $\Idx\sm n$, $\Idx\term n$, or $\Idx\klt{n,\Phi}$.
	\begin{enumerate}
		\item $\Idx{}n\subseteq\Idx{}{n+1}$.
		\item If $m\in\Idx{}n$ and $m'\in\Idx{}{n'}$, then $\lcm{m,m'}\in\Idx{}{n+n'}$.
		\end{enumerate}
\end{lem}

\begin{proof}
	(1) Let $(X,B)$ be a Calabi--Yau pair with index $m$.
		If we take $I(n)=\Idx\klt{n,\Phi}$, then we assume that $(X,B)$ is klt with coefficients in $\Phi$; if $I(n)=I_\term(n)$, we assume that $X$ is terminal and $B=0$; and if $I(n)=I_\sm(n)$, we assume that $X$ is smooth and $B=0$.
		Let $E$ be an elliptic curve.
		Then the product $(X\times E,B\times E)$ is a Calabi--Yau pair with index $m$, and it has the same coefficients and singularity type as $(X,B)$.
	
	(2) Let $(X,B)$, $(X',B')$ be Calabi--Yau pairs with index $m,m'$ respectively.
		Then the product $(X\times X', B\times X'+X\times B')$ is a Calabi--Yau pair
		with index $\lcm{m,m'}$.
\end{proof}

\subsection{Automorphisms of varieties with trivial canonical divisors}

Let $X$ be a Calabi--Yau variety with $K_X\lineq 0$, and let $G\subseteq\Aut X$ be a finite subgroup.
Then $G$ acts on the $1$-dimensional space $H^0\paren{X,K_X}$, so a group homomorphism $\alpha\colon G\to\Aut{H^0\paren{X,K_X}}\simeq\Comp^\times$ is induced.
The image of $\alpha$ is a finite cyclic group, and its order $\order{\Image\alpha}$ is called the \newterm{index} of $G$.
For an automorphism $g\in\Aut X$ of finite order, the \newterm{index} of $g$ is defined as the index of the group $\gengp g$.
In other words, if $g$ is an automorphism of $X$ of finite order with index $m$, then
\[
	g^*\omega=\zeta_m\omega\quad\text{for $\omega\in H^0\paren{X,K_X}$,}
\]
where $\zeta_m$ is a primitive $m$-th root of unity.
Note that the index of $G$ is the largest index of its elements.

Let $X$ be a normal variety, and let $G\subseteq\Aut X$ be a finite subgroup.
Let $Y=X/G$ be the quotient variety with quotient morphism $f\colon X\to Y$.
Note that $f$ is Galois.
Then by the Hurwitz formula \cite[(2.41.6)]{kollar}, there exists a (canonically determined) $\Quot$-divisor $B\ge0$ with standard coefficients such that $K_X\lineq_\Quot f^*\paren{K_Y+B}$.
We write $X/G=\paren{Y,B}$ for this situation.

\begin{prop}
	\label{quot and index}
	Let $X$ be a klt Calabi--Yau variety with $K_X\lineq0$, and let $G\subseteq\Aut X$ be a finite subgroup with index $m$.
	Then the quotient $X/G=\paren{Y,B}$ is a klt Calabi--Yau pair with standard coefficients and index $m$.
\end{prop}

\begin{proof}
	The pair $\paren{Y,B}$ is Calabi--Yau with standard coefficients by \cite[(2.41.6)]{kollar}.
	Furthermore, it is klt by \cite[Corollary 2.43]{kollar}.
	It remains to show that $(Y,B)$ has index $m$.
	Since for any integer $l\ge0$, we have $H^0(X,lK_X)^G=H^0(Y,\rounddown{l(K_Y+B)})$, it follows that $l(K_Y+B)\lineq0$ if and only if $H^0(Y,\rounddown{l(K_Y+B)})\ne0$, which is equivalent to $H^0(X,lK_X)^G\ne0$, and hence $m\mid l$.
	Therefore, we conclude that $(Y,B)$ has index $m$.
\end{proof}

\begin{lem}
	\label{divisor of index}
	Let $n\ge1$ be an integer.
	If $m\in\Idxstand n$ and $m_0\mid m$, then $m_0\in\Idxstand n$.
\end{lem}

\begin{proof}
	Let $(X,B)$ be a Calabi--Yau pair with standard coefficients and index $m$.
	Let $\tilde X\to X$ be an index-1 cover of $(X,B)$ (see \cite[Corollary 2.51]{kollar}).
	Note that the cyclic group $\mu_m$ acts on $\tilde X$ and its index is $m$.
	Consider the subgroup $\mu_{m_0}\subseteq\mu_m$.
	By \cref{quot and index}, the quotient $\tilde X/\mu_{m_0}=(Y,B_Y)$ is a klt Calabi--Yau pair with standard coefficients and index $m_0$.
\end{proof}

\subsection{Beauville--Bogomolov decomposition}

In this subsection, we introduce a decomposition theorem of smooth Calabi--Yau varieties, called the Beauville--Bogomolov decomposition.

\begin{defi}
	A \newterm{strict Calabi--Yau variety} is a simply connected smooth projective variety $X$ of dimension $\ge2$ such that $K_X\lineq0$ and $H^0\paren{X,\difsh Xi}=0$ for $0<i<\dim X$.
	A \newterm{holomorphic symplectic variety} is a simply connected smooth projective variety $X$ of dimension $\ge2$ such that there exists $\sigma\in H^0\paren{X,\difsh X2}$ that is everywhere non-degenerate.
\end{defi}

\begin{thm}[{\cite[\S 1]{beauville}}]
	\label{decomposition}
	Let $X$ be a smooth Calabi--Yau variety.
	Then there is a finite \'etale cover $f\colon\tilde X\to X$ such that $\tilde X=\prod_iY_i\times\prod_jZ_j\times A$ is a product of strict Calabi--Yau varieties $Y_i$, holomorphic symplectic varieties $Z_j$, and an abelian variety $A$.
\end{thm}

For a variety $\tilde X$ as in \cref{decomposition}, its automorphism group can be computed as follows:

\begin{prop}
	[{\cite[\S 3]{beauville}}]
	\label{aut of tilde X}
	Let $e_1,\dotsc,e_r\ge1$ be integers, $Y_1,\dotsc,Y_r$ be non-isomorphic varieties that are either strict Calabi--Yau or holomorphic symplectic, and $A$ be an abelian variety.
	Let $X=\prod_i Y_i^{e_i}\times A$.
	Then it holds that
	\[
		\Aut X=\prod_{i=1}^r\paren{\Aut{Y_i}^{e_i}\rtimes S_{e_i}}\times\Aut A,
	\]
	where $S_e$ is the symmetric group.
\end{prop}

\subsection{Holomorphic Lefschetz formula}

In this subsection, we introduce the holomorphic Lefschetz formula,
and deduce a property of fixed-point-free automorphisms
of strict Calabi--Yau varieties and holomorphic symplectic varieties.
The following is a special case
of the holomorphic Lefschetz formula \cite[Theorem (4.6)]{atiyah-singer}:

\begin{thm}
	\label{lefschetz for no fixed point}
	Let $X$ be a smooth projective variety of dimension $n$.
	Let $g$ be an automorphism of $X$ of finite order without fixed points.
	Then it holds that
	\[
		\sum_{i=0}^n\paren{-1}^i\trace\paren{\restr{g^*}{H^i\paren{X,\mathcal O_X}}}=0.
	\]
\end{thm}

Applying \cref{lefschetz for no fixed point} to strict Calabi--Yau or holomorphic symplectic varieties,
we have the following corollary.

\begin{cor}
	\label{index of aut of no fixed points}
	Let $X$ be a strict Calabi--Yau or holomorphic symplectic variety of dimension $n$.
	Let $g$ be an automorphism of $X$ of finite order, with index $m$ and without fixed points.
	\begin{enumerate}
		\item Assume $X$ is strict Calabi--Yau and $n$ is even.
			Then $m=2$.
		\item Assume $X$ is strict Calabi--Yau and $n$ is odd.
			Then $m=1$.
		\item Assume $X$ is holomorphic symplectic.
			Then $n/2\equiv -1\pmod m$, and in particular, $m\le n/2+1$.
	\end{enumerate}
\end{cor}

\begin{proof}
	First show (1) and (2), so assume $X$ is strict Calabi--Yau.
	Since the index of $g$ is $m$, the action of $g^*$ on the $1$-dimensional space
	$H^0(X,K_X)$ is given by multiplication by $\zeta_m$, a primitive $m$-th root
	of unity.
	Then the action of $g^*$ on $H^n(X,\mathcal O_X)$ is given
	by multiplication by $\zeta_m^{-1}$ by Serre duality.
	Thus, \cref{lefschetz for no fixed point} implies that
	\[
		1+(-1)^n\zeta_m^{-1}=0.
	\]
	This shows that $m=2$ if $n$ is even, and $m=1$ if $n$ is odd.

	Next we prove (3).
	Write $n=2r$, and suppose that the action of $g^*$
	on the $1$-dimensional space $H^0\paren{X,\Omega_X^2}$ is given
	by multiplication by $a\in\Comp$.
	Then the action of $g^*$ on $H^0(X,\Omega_X^{2i})$ is given by multiplication
	by $a^i$, thus the action of $g^*$ on $H^{2i}(X,\mathcal O_X)$ is given by
	multiplication by $\bar a^i$, the complex conjugate, by Hodge symmetry.
	Therefore \cref{lefschetz for no fixed point} gives that
	\[
		\sum_{i=0}^r \bar a^i=0,
	\]
	which means that $a^{r+1}=1$.
	Since the index of $g$ is $m$, it follows that $a^r$ is a primitive $m$-th root of unity.
	Therefore, $a=a^{-r}$ is also a primitive $m$-th root of unity.
	In particular, $r\equiv-1\pmod m$.
\end{proof}

A similar result is known for automorphisms of abelian varieties.
\label{euler}
To elaborate, we introduce the \newterm{Euler function} $\varphi$.
For an integer $m\ge1$, define $\varphi\paren m$ as the order of the multiplicative group $\multgp{\paren{\Zahlen/m\Zahlen}}$.
If we write $m=p_1^{e_1}\dotsm p_r^{e_r}$ with distinct prime numbers $p_i$ and $e_i\ge1$, then we can see that
\[
	\varphi\paren m=\prod_{i=1}^r\paren{p_i^{e_i}-p_i^{e_i-1}}.
\]
In particular, $\varphi\paren m$ is either $1$ or an even number, and if $m$ and $m'$ are relatively prime, then $\varphi\paren{mm'}=\varphi\paren m\varphi\paren{m'}$.

\begin{prop}[{cf.~\cite[Remark following Proposition 5]{beauville}}]
	\label{abelian phi<2n}
	Let $X$ be an abelian variety of dimension $n$, and let $G\subseteq\Aut X$ be a finite subgroup with index $m$.
	Then it holds that $\varphi\paren m\le 2n$.
\end{prop}

\begin{proof}
	Take an element $g\in G$ with index $m$.
	Let $l$ be the order of $g^*\rvert_{H^1(X,\Zahlen)}$.
	Then the $l$-th cyclotomic polynomial $\Phi_l$ divides
	the minimal polynomial $F$ of $g^*\rvert_{H^1(X,\Zahlen)}$.
	Since the triviality of $(g^*)^l\rvert_{H^1(X,\Comp)}$ implies the triviality
	of $(g^*)^l\rvert_{H^0(X,K_X)}$,
	we see that $m\mid l$.
	Thus,
	\[
		\varphi(m)\le\varphi(l)=\deg\Phi_l\le\deg F
		\le\dim H^1(X,\Comp)=2n.\qedhere
	\]
\end{proof}

Combining \cref{index of aut of no fixed points,abelian phi<2n}, we get the following:

\begin{cor}
	\label{free aut phi<2n}
	Let $X$ be a strict Calabi--Yau, holomorphic symplectic, or abelian variety of dimension $n$.
	Let $g$ be an automorphism of $X$ of finite order with index $m$.
	Assume that $g$ has no fixed point.
	Then it holds that $\varphi\paren m\le2n$.
\end{cor}

\begin{proof}
	If $X$ is strict Calabi--Yau, then $m=1,2$ by \cref{index of aut of no fixed points},
	so $\varphi(m)=1\le2n$.
	If $X$ is holomorphic symplectic, then $m\le n/2+1$ by \cref{index of aut of no fixed points} again,
	and thus $\varphi(m)\le m\le n/2+1\le 2n$.
	If $X$ is abelian, the result follows from \cref{abelian phi<2n}.
\end{proof}

\subsection{Weighted projective spaces}

In this subsection, we give a brief overview of fundamental facts about weighted projective spaces, which are essential for \cref{existence of pairs}.
Refer to \cite{iano-fletcher} for further details.

Let $a_0,\dotsc,a_n\ge1$ be integers.
The \newterm{weighted projective space} $X=\wproj{a_0,\dotsc,a_n}$ is defined as
\[
	X=\wproj{a_0,\dotsc,a_n}
	=\paren{\affsp{n+1}\setminus\finset0}/\Gm,
\]
where the multiplicative group $\Gm$ acts on $\affsp{n+1}\setminus\finset0$ by
\[
	t\paren{x_0,\dotsc,x_n}=\paren{t^{a_0}x_0,\dotsc,t^{a_n}x_n}
	\quad (\text{for $t\in\Gm$}).
\]
The \newterm{canonical projection} is the induced map $\pi\colon\affsp{n+1}\setminus\finset0\to X$.
The weighted projective space $\wproj{a_0,\dotsc,a_n}$ is said to be \newterm{well-formed} if
	\[
		\gcd\paren{a_0,\dotsc,\widehat {a_i},\dotsc,a_n}=1
	\]
for each $i$.
We write $\wproj{a_0^{\paren{b_0}},\dotsc,a_r^{\paren{b_r}}}$ for the weighted projective space with the weight $a_i$ repeated $b_i$ times.

Note that a weighted projective space $X=\wproj{a_0,\dotsc,a_n}$ can be written as the quotient of $\projsp n$ by a finite group (\cite[5.12]{iano-fletcher}).
Consequently, $X$ has only quotient singularities.
In particular, $X$ is normal and $\Quot$-factorial.

Let $x_0,\dotsc,x_n$ be the homogeneous coordinates of $X=\wproj{a_0,\dotsc,a_n}$, and let $U_i=\hypsurf{x_i\ne0}$ be the open subsets of $X$.
The \newterm{coordinate chart} on $U_i$ is the morphism
\[
	\affsp n\to U_i,\quad
	\paren{z_0,\dotsc,\widehat{z_i},\dotsc,z_n}\mapsto[z_0:\dotsb:1:\dotsb:z_n].
\]
Then $U_i$ is isomorphic to the quotient $\affsp n/\mu_{a_i}$, where the cyclic group $\mu_{a_i}=\gengp\zeta$ acts on $\affsp n$ by
\[
	\zeta\paren{z_0,\dotsc,\widehat{z_i},\dotsc,z_n}=\paren{\zeta^{a_0}z_0,\dotsc,\widehat{z_i},\dotsc,\zeta^{a_n}z_n}
\]
with $\zeta$ a primitive $a_i$-th root of unity (\cite[5.3]{iano-fletcher}).

\begin{lem}
	\label{well-formed}
	Let $X=\wproj{a_0,\dotsc,a_n}$ be a weighted projective space.
	Then the following are equivalent:
	\begin{enumerate}
		\item The weighted projective space $X=\wproj{a_0,\dotsc,a_n}$ is well-formed.
		\item For each $i$, the action of $\mu_{a_i}$ on the coordinate chart $\affsp n$ on $U_i$ is free in codimension $1$.
	\end{enumerate}
\end{lem}

\begin{proof}
	Let $i\in\finset{0,\dotsc,n}$.
	Let $\zeta^b\in\mu_{a_i}$ ($1\le b<a_i$) with $\zeta$ a primitive $a_i$-th root of unity.
	Then the fixed locus $\paren{\affsp n}^{\zeta^b}$ is given by
	\begin{align*}
		(\affsp n)^{\zeta^b}
		&=\sset{(z_0,\dotsc,\widehat{z_i},\dotsc,z_n)}{z_j=\zeta^{ba_j}z_j}\\
		&=\sset{(z_0,\dotsc,\widehat{z_i},\dotsc,z_n)}{\text{$z_j=0$ if $a_i\nmid ba_j$}}.
	\end{align*}
	Let $c=\gcd(a_i,b)$, and write $a_i=a'_ic$ and $b=b'c$.
	Note that $a'_i$ and $b'$ are relatively prime.
	Then the fixed locus is given by
	\[
		(\affsp n)^{\zeta^b}
		=\sset{(z_0,\dotsc,\widehat{z_i},\dotsc,z_n)}{\text{$z_j=0$ if $a'_i\nmid a_j$}}.
	\]

	First, assume (1), and let $i\in\finset{0,\dotsc,n}$.
	We will show that the action of $\mu_{a_i}$ on $\affsp n$ is free in codimension $1$.
	If $a_i=1$, then the action is clearly free, so we assume $a_i\ge2$.
	Let $1\le b<a_i$, and it suffices to show that
	\[
		(\affsp n)^{\zeta^b}=\sset{(z_0,\dotsc,\widehat{z_i},\dotsc,z_n)}{\text{$z_j=0$ if $a'_i\nmid a_j$}}
	\]
	has codimension $\ge2$.
	Note that $a'_i\ge2$ since $a_i\ge2$.
	Since $\gcd(a_0,\dotsc,\widehat{a_i},\dotsc,a_n)=1$, there is $j_0\ne i$ such that $a'_i\nmid a_{j_0}$.
	Moreover, since $\gcd(a_0,\dotsc,\widehat{a_{j_0}},\dotsc,a_n)=1$ and $a'_i\mid a_i$, there is $j_1\ne i,j_0$ such that $a'_i\nmid a_{j_1}$.
	Therefore, we obtain that
	\[
		(\affsp n)^{\zeta^b}\subseteq\finset{z_{j_0}=z_{j_1}=0}.
	\]
	In particular, $(\affsp n)^{\zeta^b}$ has codimension $\ge2$.

	Conversely, assume (2).
	Let $j_0\in\finset{0,\dotsc,n}$ and let $g=\gcd(a_0,\dotsc,\widehat{a_{j_0}},\dotsc,a_n)$.
	It suffices to show that $g=1$.
	Suppose $g\ge2$.
	Take $i\in\finset{0,\dotsc,n}\setminus\finset{j_0}$ and consider the action of $\mu_{a_i}=\gengp\zeta$ on the coordinate chart $\affsp n$.
	We can write $a_i=bg$ for some $1\le b<a_i$.
	Then, by the above argument, the fixed locus is given by
	\begin{align*}
		(\affsp n)^{\zeta^b}
		&=\sset{(z_0,\dotsc,\widehat{z_i},\dotsc,z_n)}{\text{$z_j=0$ if $g\nmid a_j$}}\\
		&\supseteq\finset{z_{j_0}=0}.
	\end{align*}
	This contradicts (2).
	Therefore, we conclude that $g=1$.
\end{proof}

In the rest of this subsection, we describe a criterion for determining whether a pair
\[
	\paren{\wproj{a_0,\dotsc,a_n},B}
\]
is klt.
It is used in \cite[Proof of Theorem 3.3]{esser-totaro-wang-cal}, but we elaborate on it here for the convenience of the readers.

\begin{defi}
	Let $\paren{X,B}$ be a pair with $X=\wproj{a_0,\dotsc,a_n}$ a weighted projective space.
	The \newterm{affine cone} of $\paren{X,B}$ is the pair $\paren{\affsp{n+1},B'}$, where $B'$ is the $\Quot$-divisor on $\affsp{n+1}$ such that $\restr{B'}{\affsp{n+1}\setminus\finset0}=\pi^*B$ with $\pi\colon\affsp{n+1}\setminus\finset0\to X$ the canonical projection.
\end{defi}

\begin{prop}
	\label{klt for wps}
	Let $\paren{X,B}$ be a pair with $X=\wproj{a_0,\dotsc,a_n}$ a well-formed weighted projective space.
	Then $\paren{X,B}$ is klt if and only if its affine cone $\paren{\affsp{n+1},B'}$ is klt outside the origin.
\end{prop}

\begin{proof}
	Let $x_0,\dotsc,x_n$ be the homogeneous coordinates of $X$, and let $U_i$ be the open subsets $\finset{x_i\ne0}$ of $X$.
	Let $y_0,\dotsc,y_n$ be the coordinates of $\affsp{n+1}$, and let $V_i=\hypsurf{y_i\ne0}$ in $\affsp{n+1}$.
	Note that $V_i=\pi^{-1}\paren{U_i}$ where $\pi\colon\affsp{n+1}\setminus\finset0\to X$ is the canonical projection.
	The open subvariety $V_i$ is isomorphic to $\paren{\Gm\times\affsp n}/\mu_{a_i}$, where $\mu_{a_i}$ acts on $\Gm\times\affsp n$ by
	\[
		\zeta\paren{t,z}=\paren{\zeta^{-1}t,\zeta\paren z}.
	\]
	The quotient map $\Gm\times\affsp n\to V_i$ is given by
	\[
		\paren{t,z_0,\dotsc,\widehat{z_i},\dotsc,z_n}
		\mapsto
		\paren{t^{a_0}z_0,\dotsc,t^{a_{i-1}}z_{i-1},t^{a_i},t^{a_{i+1}}z_{i+1},\dotsc,t^{a_n}z_n}.
	\]
	Note that the action of $\mu_{a_i}$ on $\Gm\times\affsp n$ is free since $\mu_{a_i}$ acts freely on the $\Gm$-factor.
	Moreover, since $X=\wproj{a_0,\dotsc,a_n}$ is well-formed, the action of $\mu_{a_i}$ on $\affsp n$ is free in codimension $1$ by \cref{well-formed}.
	We have the following commutative diagram:
	\[
	\begin{tikzcd}
		\Gm\times\affsp n\arrow[r,"/\mu_{a_i}"]\arrow[d]
%		\Gm\times\affsp n\arrow[r]\arrow[d]
			&V_i\arrow[d,"\pi"]\\
		\affsp n\arrow[r,"/\mu_{a_i}"']
%		\affsp n\arrow[r]
			&U_i
	\end{tikzcd}
	\]
	
	We now prove the proposition.
	The pair $\paren{X,B}$ is klt if and only if $\paren{U_i,\restr B{U_i}}$ is klt for each $i$.
	By the diagram above, this is equivalent to that $\paren{V_i,\restr{B'}{V_i}}$ is klt for each $i$, that is, the affine cone $\paren{\affsp{n+1},B'}$ is klt outside the origin.
\end{proof}

\section{Indices of Calabi--Yau pairs of lower dimensions}

In this section, we show that \cref{index set conj} is true in dimension $n\le3$.
First, we prove the second inclusion in \cref{intro main thm}, following \cite[Proof of Corollary 4.1]{esser-totaro-wang-cal}:

\begin{prop}
	\label{klt st subseteq term}
	Let $n\ge1$ be an integer.
	Then it holds that
	\[
		\Idxstand{n-1}\subseteq\Idx\term n.
	\]
\end{prop}

\begin{proof}
	Let $\paren{X,B}$ be a klt Calabi--Yau pair of dimension $n-1$
	with standard coefficients and index $m$.
	It suffices to show that $m\in\Idx\term n$.
	Consider an index-1 cover $f\colon\tilde X\to X$ of $\paren{X,B}$ (see \cite[Corollary 2.51]{kollar}).
	Note that $\tilde X$ is canonical, and  the cyclic group $\mu_m$ acts on $\tilde X$.
	Consider an elliptic curve $E$, and let $\mu_m$ act on $E$ by translation of order $m$.
	Then the quotient $Y=\paren{\tilde X\times E}/\mu_m$ is a canonical Calabi--Yau variety of dimension $n$ with index $m$.
	Taking a terminalization of $Y$, it can be seen that $m$ is the index of a terminal Calabi--Yau variety of dimension $n$.
\end{proof}

We now show that \cref{index set conj} is true in dimension $n\le3$,
and we also explicitly determine the sets $I_\sm(n)$, $I_\term(n)$ and $\Idx\klt{n,\Phist}$ in these dimensions.
It was essentially proved by Machida and Oguiso \cite{machida-oguiso}.

\begin{prop}
	\label{index set in low dimension}
	Let $n\in\finset{1,2,3}$.
	Then it holds that
	\[
		\Idxstand{n-1}=\Idx\sm n=\Idx\term n.
	\]
	Specifically, it holds that
	\begin{align*}
		\Idxstand 0=\Idx\sm1=\Idx\term1&=\finset1,\\
		\Idxstand 1=\Idx\sm2=\Idx\term2&=\finset{1,2,3,4,6},\\
		\Idxstand 2=\Idx\sm 3=\Idx\term3&=\sset{m\in\Zahlen_{\ge1}}{\varphi\paren m\le20}\setminus\finset{60}.
	\end{align*}
\end{prop}

\begin{proof}
	By definition, $\Idxstand0=\finset1$.
	Both smooth Calabi--Yau curves and terminal Calabi--Yau curves are smooth elliptic curves, and their indices are $1$.
	
	Let $\paren{X,B}$ be a klt Calabi--Yau curve with standard coefficients.
	Then $X$ is either a smooth elliptic curve or a smooth rational curve.
	If $X$ is elliptic, then $B=0$ and the index of $\paren{X,B}$ is $1$.
	If $X=\projsp1$, we can easily see that $\paren{X,B}$ is one of the following:
	\begin{align*}
		&\Paren{\projsp1,\frac12\paren{P_1+P_2+P_3+P_4}},
		&&\Paren{\projsp1,\frac23\paren{P_1+P_2+P_3}},\\
		&\Paren{\projsp1,\frac12P_1+\frac34\paren{P_2+P_3}},
		&&\Paren{\projsp1,\frac12P_1+\frac23P_2+\frac56P_3},
	\end{align*}
	where the $P_i$ are general points.
	Therefore the index of $\paren{X,B}$ is $2$, $3$, $4$ or $6$.
	This shows that $\Idxstand1=\finset{1,2,3,4,6}$.
	
	Since terminal surfaces are smooth, we see that $\Idx\sm2=\Idx\term2$.
	Moreover, the well-known classification of smooth projective surfaces
	\cite[Chapter VIII]{beauville-book} shows that
	$\Idx\sm2=\finset{1,2,3,4,6}$.
	
	By \cite[Main Theorem 3]{machida-oguiso}, for each integer $m\ge1$ with $\varphi\paren m\le20$ and $m\ne60$, there exist a K3 surface $X$ and an automorphism $g$ of $X$ of finite order with index $m$.
	This together with \cref{quot and index} shows that
	\[
		\sset{m\in\Zahlen_{\ge1}}{\varphi\paren m\le20}\setminus\finset{60}\subseteq\Idxstand2.
	\]
	Moreover, by \cite[Corollary 5]{machida-oguiso}, we have that
	\[
		\Idx\sm3
		=\Idx\term3
		=\sset{m\in\Zahlen_{\ge1}}{\varphi\paren m\le20}\setminus\finset{60}.
	\]
	Now we conclude the proof by \cref{klt st subseteq term}.
\end{proof}

\section{Existence of Calabi--Yau pairs with small indices}
\label{existence of pairs}

In this section, we prove \cref{intro small index}, which plays an important role in \cref{indices of sm cy}.
First, we prove two propositions on the existence of Calabi--Yau pairs with given indices:

\begin{prop}
	\label{index prime}
	Let $m\ge5$ be an odd number.
	\begin{enumerate}
		\item Assume $m\equiv1\pmod4$.
			Then $m\in\Idxstand{\paren{m+3}/4}$.
		\item Assume $m\equiv3\pmod4$.
			Then $m\in\Idxstand{\paren{m+1}/4}$.
	\end{enumerate}
\end{prop}

\begin{proof}
	(1) Let $n=\paren{m+3}/4$.
	We define
	\[
		X=\wproj{4^{\paren{n-2}},2,1,1},\quad
		B=\frac{m-1}m\Paren{\sum_{i=0}^{n-3}H_i+H_n+H},
	\]
	where $H_i=\hypsurf{x_i=0}$ and $H=\hypsurf{x_0+\dotsb+x_{n-3}+x_{n-2}^2+x_{n-1}^4+x_n^4=0}$ are prime divisors on $X$, and $x_0,\dotsc,x_n$ are the homogeneous coordinates of $X$.
	Note that the weighted projective space $X$ is well-formed and the pair $\paren{X,B}$ has standard coefficients.

	Assume that $K_X+B$ is $\Quot$-linearly equivalent to $\mathcal O_X(d)$.
	Then
	\begin{align*}
		d
		&=-(4(n-2)+2+1+1)+\frac{m-1}m(4(n-2)+1+4)\\
		&=1-\frac{4n-3}m\\
		&=0.
	\end{align*}
	Therefore, $K_X+B\lineqQ0$.
	Moreover, from the definition of $B$, the index of $(X,B)$ is $m$.

	Finally, we prove that $(X,B)$ is klt.
	By \cref{klt for wps}, it suffices to show that the affine cone $(\affsp{n+1},B')$ is klt outside the origin.
	Here
	\[
		B'=\frac{m-1}m\Paren{\sum_{i=0}^{n-3}H'_i+H'_n+H'},
	\]
	where the prime divisors $H'_i$ and $H'$ on $\affsp{n+1}$ are defined by the same equations as those defining $H_i$ and $H$, respectively.
	In fact, we prove the stronger claim that $B'$ has simple normal crossing support outside the origin.
	Since the support is clearly simple normal crossing outside $\bigcap_{i=0}^{n-3}H'_i$,
	it suffices to show that the divisor
	\[
		\finset{x_n=0}+\finset{x_{n-2}^2+x_{n-1}^4+x_n^4=0}\quad\text{on}\quad\affsp3
	\]
	is simple normal crossing outside the origin.
	This can be seen by checking that the surface $\finset{x_{n-2}^2+x_{n-1}^4+x_n^4=0}$ in $\affsp3$ and the curve $\finset{x_{n-2}^2+x_{n-1}^4=0}$ in $\affsp2$ are smooth outside the origin.

	(2) Let $n=\paren{m+1}/4$.
	We define
	\[
		X=\wproj{4^{\paren{n-2}},3,2,1},\quad
		B=\frac{m-1}m\Paren{\sum_{i=0}^{n-2}H_i+H},
	\]
	where $H_i=\hypsurf{x_i=0}$ and $H=\hypsurf{x_0+\dotsb+x_{n-3}+x_{n-2}x_{n}+x_{n-1}^2+x_n^4=0}$ are prime divisors on $X$, and $x_0,\dotsc,x_n$ are the homogeneous coordinates of $X$.
	Note that the weighted projective space $X$ is well-formed and the pair $(X,B)$ has standard coefficients.

	Assume that $K_X+B$ is $\Quot$-linearly equivalent to $\mathcal O_X(d)$.
	Then 
	\begin{align*}
		d
		&=-(4(n-2)+3+2+1)+\frac{m-1}m(4(n-2)+3+4)\\
		&=1-\frac{4n-1}m\\
		&=0.
	\end{align*}
	Therefore, $K_X+B\lineqQ0$.
	Moreover, from the definition of $B$, the index of $(X,B)$ is $m$.

	Finally, we prove that $(X,B)$ is klt.
	By \cref{klt for wps}, it suffices to show that the affine cone $(\affsp{n+1},B')$ is klt outside the origin.
	Here
	\[
		B'=\frac{m-1}m\Paren{\sum_{i=0}^{n-2}H'_i+H'},
	\]
	where the prime divisors $H'_i$ and $H'$ on $\affsp{n+1}$ are defined by the same equations as those defining $H_i$ and $H$, respectively.
	In fact, we prove the stronger claim that $B'$ has simple normal crossing support outside the origin.
	Since the support is clearly simple normal crossing outside $\bigcap_{i=0}^{n-3}H'_i$, it suffices to show that the divisor
	\[
		\finset{x_{n-2}=0}+\finset{x_{n-2}x_n+x_{n-1}^2+x_n^4=0}\quad\text{on}\quad\affsp3
	\]
	is simple normal crossing outside the origin.
	This can be seen by checking that the surface $\finset{x_{n-2}x_n+x_{n-1}^2+x_n^4=0}$ in $\affsp3$ and the curve $\finset{x_{n-1}^2+x_n^4=0}$ in $\affsp2$ are smooth outside the origin.
\end{proof}

\begin{prop}
	\label{index power of prime}
	Let $m,e\ge2$ be integers.
	Then $m^e\in\Idxstand{m+e-3}$.
\end{prop}

\begin{proof}
	Let
	\[
		X=\wproj{\paren{m-1}^{\paren{e-1}}, 1^{\paren{m-1}}},\quad
		B=\frac{m^e-1}{m^e}H+\sum_{i=0}^{e-1}\frac{m^{i+1}-1}{m^{i+1}}H_i,
	\]
	where $H_i=\hypsurf{x_i=0}$ and $H=\hypsurf{x_0+\dotsb+x_{e-2}+x_{e-1}^{m-1}+\dotsb+x_{e+m-3}^{m-1}=0}$ are prime divisors on $X$, and $x_0,\dotsc,x_{e+m-3}$ are the homogeneous coordinates of $X$.
	Note that the weighted projective space $X$ is well-formed and the pair $(X,B)$ has standard coefficients.

	Assume that $K_X+B$ is $\Quot$-linearly equivalent to $\mathcal O_X(d)$.
	Then
	\begin{align*}
		d
		&=-((m-1)(e-1)+(m-1))\\
			&\quad+\frac{m^e-1}{m^e}(m-1)+\sum_{i=0}^{e-2}\frac{m^{i+1}-1}{m^{i+1}}(m-1)+\frac{m^e-1}{m^e}\\
		&=-\frac{m-1}{m^e}-\sum_{i=0}^{e-2}\frac{m-1}{m^{i+1}}+\frac{m^e-1}{m^e}\\
		&=-(m-1)\sum_{i=0}^{e-1}\frac1{m^{i+1}}+\frac{m^e-1}{m^e}\\
		&=0.
	\end{align*}
	Therefore, $K_X+B\lineqQ0$.
	Moreover, from the definition of $B$, the index of $(X,B)$ is $m^e$.

	Finally, we prove that $(X,B)$ is klt.
	By \cref{klt for wps}, it suffices to show that the affine cone $(\affsp{m+e-2},B')$ is klt outside the origin.
	Here
	\[
		B'=\frac{m^e-1}{m^e}H'+\sum_{i=0}^{e-1}\frac{m^{i+1}-1}{m^{i+1}}H'_i,
	\]
	where the prime divisors $H'_i$ and $H'$ on $\affsp{m+e-2}$ are defined by the same equations as those defining $H_i$ and $H$, respectively.
	In fact, we prove the stronger claim that $B'$ has simple normal crossing support outside the origin.
	Since the support is clearly simple normal crossing outside $\bigcap_{i=0}^{e-2}H'_i$, it suffices to show that the divisor
	\[
		\finset{x_{e-1}=0}+\finset{x_{e-1}^{m-1}+\dotsb+x_{e+m-3}^{m-1}=0}
		\quad\text{on}\quad\affsp{m-1}
	\]
	is simple normal crossing outside the origin.
	This can be seen by checking that the hypersurfaces $\finset{x_{e-1}^{m-1}+\dotsb+x_{e+m-3}^{m-1}=0}$ in $\affsp{m-1}$ and $\finset{x_{e}^{m-1}+\dotsb+x_{e+m-3}^{m-1}=0}$ in $\affsp{m-2}$ are smooth outside the origin.
\end{proof}

We now prepare two inequalities on the dimension $m+e-3$ that appears in \cref{index power of prime}:

\begin{lem}
	\label{inequality}
	Let $m,e\ge2$ be integers, and let $n=\paren{m^e-m^{e-1}}/2$.
	\begin{enumerate}
		\item Assume $\paren{m,e}\ne\paren{2,2},\paren{2,3}$.
			Then it holds that
			\[
				m+e-3\le n-1.
			\]
		\item Assume $m\ge3$ and $\paren{m,e}\ne\paren{3,2}$.
			Then it holds that
			\[
				m+e-3\le n-3.
			\]
	\end{enumerate}
\end{lem}

\begin{proof}
	It suffices to show the following:
	\begin{enumerate}
		\item $m^e-m^{e-1}-2m-2e+4\ge0$ if $\paren{m,e}\ne\paren{2,2},\paren{2,3}$, and
		\item $m^e-m^{e-1}-2m-2e\ge0$ if $m\ge3$ and $\paren{m,e}\ne\paren{3,2}$.
	\end{enumerate}
	
	First show (2) by induction on $e$.
	If $e=2$, then as $m\ge4$, we have
	\[
		m^e-m^{e-1}-2m-2e=(m+1)(m-4)\ge0.
	\]
	If $m=e=3$, then $m^e-m^{e-1}-2m-2e=6\ge0$.
	Now assume $e\ge3$ if $m\ge4$, and $e\ge4$ if $m=3$.
	Using the induction hypothesis, we see that
	\begin{align*}
		m^e-m^{e-1}-2m-2e
		&=m(m^{e-1}-m^{e-2})-2m-2e\\
		&\ge m(2m+2(e-1))-2m-2e\\
		&=2m\paren{m-2}+2\paren{m-1}e\\
		&\ge0.
	\end{align*}
	
	Next we prove (1).
	By the inequality (2), we may assume that $m=2$ or $\paren{m,e}=\paren{3,2}$.
	If $m=2$, then as $e\ge4$,
	\[
		m^e-m^{e-1}-2m-2e+4=2^{e-1}-2e\ge0.
	\]
	If $\paren{m,e}=\paren{3,2}$, then $m^e-m^{e-1}-2m-2e+4=0$.
\end{proof}

Finally we prove \cref{intro small index}:

\begin{thm}
	\label{phi<2n and index}
	Let $n\ge3$ be an integer.
	Then
	\[
		\sset{m\in\Zahlen_{\ge1}}{\varphi\paren m\le2n}
		\subseteq
		\Idxstand{n-1},
	\]
	where $\varphi$ is the Euler function (see \cref{euler}).
\end{thm}

\begin{proof}
	We prove the theorem by induction on $n$.
	First consider the case when $n=3$.
	By \cref{index set in low dimension}, we have that
	\[
		\Idxstand2
		=\sset{m\in\Zahlen_{\ge1}}{\varphi\paren m\le20}\setminus\finset{60}.
	\]
	Since $\varphi\paren{60}=16>6$, we conclude that $\sset{m\ge1}{\varphi\paren m\le6}\subseteq\Idxstand2$.
	
	Let $n\ge4$, and let $m\ge1$ with $\varphi\paren m\le2n$.
	We will show that $m\in\Idxstand{n-1}$.
	If $\varphi(m)<2n$, then $\varphi(m)\le2(n-1)$, and thus
	\[
		m\in\Idxstand{n-2}\subseteq\Idxstand{n-1}
	\]
	by the induction hypothesis and \cref{properties of Idx}.
	Therefore we assume $\varphi(m)=2n$ in the following.
	We write $m=p_1^{e_1}\dotsm p_r^{e_r}$, where $p_1<\dotsb<p_r$ are prime numbers and $e_i\ge1$.
	
	First, assume $r=1$ and $e_1=1$, which means that $m$ is prime.
	Then $2n=\varphi\paren m=m-1$.
	By \cref{index prime}, we have
	\[
		m\in
		\begin{cases}
			\Idxstand{\paren{m+3}/4}&\text{if $m\equiv1\pmod4$},\\
			\Idxstand{\paren{m+1}/4}&\text{if $m\equiv3\pmod4$}.
		\end{cases}
	\]
	Since $n\ge4$, it follows that
	\[
		\frac{m+1}4\le\frac{m+3}4=\frac{n+2}2\le n-1.
	\]
	Therefore, by \cref{properties of Idx}, we conclude that $m\in\Idxstand{n-1}$.
	
	Assume $r=1$ and $e_1\ge2$, and write $m=p^e$.
	In this case, we have $2n=\varphi\paren m=p^e-p^{e-1}$,
	and since $n\ge4$, it follows that $p^e\ne 2^2, 2^3$.
	By \cref{index power of prime}, we obtain that $m=p^{e}\in\Idxstand{p+e-3}$.
	Furthermore, we have $p+e-3\le n-1$ by \cref{inequality}.
	Therefore, by \cref{properties of Idx}, we conclude that
	\[
		m=p^e\in\Idxstand{p+e-3}\subseteq\Idxstand{n-1}.
	\]

	Assume $r\ge2$ in the following.
	Let $m_1=p_1^{e_1}\dotsm p_{r-1}^{e_{r-1}}$ and $m_2=p_r^{e_r}$,
	so that $m=m_1m_2$ and $2n=\varphi\paren{m_1}\varphi\paren{m_2}$.
	First assume $\varphi\paren{m_i}\ge6$ for $i=1,2$.
	Then we can write $\varphi\paren{m_i}=2n_i$ with $n_i\ge3$,
	which implies that $n=2n_1n_2$.
	By the induction hypothesis, we obtain that $m_i\in\Idxstand{n_i-1}$.
	Moreover, we see that $n_1+n_2-2\le n-1$ since
	\[
		n-1-(n_1+n_2-2)
		=2n_1n_2-n_1-n_2+1
		=n_1n_2+(n_1-1)(n_2-1)\ge0.
	\]
	Therefore, by \cref{properties of Idx}, we conclude that
	\[
		m=m_1m_2\in\Idxstand{n_1+n_2-2}\subseteq\Idxstand{n-1}.
	\]
	
	It remains to show the case when $\varphi\paren{m_1}\le4$ or $\varphi\paren{m_2}\le4$.
	Note that, by \cref{index set in low dimension}, for $l\ge2$,
	\[
		\begin{tikzcd}[row sep=0em]
			\varphi\paren l=1\arrow[r,Leftrightarrow]&
				l=2\arrow[r,Rightarrow]&
					l\in\Idxstand1,\\
			\varphi\paren l=2\arrow[r,Leftrightarrow]&
				l=3,4,6\arrow[r,Rightarrow]&
					l\in\Idxstand1,\\
			\varphi\paren l=4\arrow[r,Leftrightarrow]&
				l=5,8,10,12\arrow[r,Rightarrow]&
					l\in\Idxstand2.
		\end{tikzcd}
	\]
	Moreover, since $p_r>p_1\ge2$, we have that $\varphi\paren{m_2}=\varphi(p_r^{e_r})\ge2$.

	First assume $\varphi\paren{m_1}=2n_1\ge6$ and $\varphi\paren{m_2}\le4$.
	In this case, $n=n_1\varphi(m_2)$.
	By the induction hypothesis, we have $m_1\in\Idxstand{n_1-1}$.
	If $\varphi\paren{m_2}=2$, then $n=2n_1$ and $m_2\in\Idxstand1$, and therefore by \cref{properties of Idx}, we obtain that
	\[
		m=m_1m_2\in\Idxstand{(n_1-1)+1}=\Idxstand{n_1}.
	\]
	Since $n_1=n/2\le n-1$, we conclude, by \cref{properties of Idx}, that
	\[
		m\in\Idxstand{n_1}\subseteq\Idxstand{n-1}.
	\]
	If $\varphi\paren{m_2}=4$, then $n=4n_1$ and $m_2\in\Idxstand2$, and therefore by \cref{properties of Idx}, we obtain that
	\[
		m=m_1m_2\in\Idxstand{(n_1-1)+2}=\Idxstand{n_1+1}.
	\]
	Since $n_1+1=n/4+1\le n-1$, we conclude, by \cref{properties of Idx}, that
	\[
		m\in\Idxstand{n_1+1}\subseteq\Idxstand{n-1}.
	\]
	
	It now suffices to show the case when $\varphi\paren{m_1}\le4$.
	First, assume that $\varphi\paren{m_1}=1$ (i.e., $m_1=2$), and write $m=2p^e$.
	If $e=1$, then $2n=\varphi\paren m=p-1$, and in particular, $p\ge11$ since $n\ge4$.
	Then, by \cref{index prime,properties of Idx}, we have
	\[
		m=2p\in
		\begin{cases}
			\Idxstand{1+(p+3)/4} & \text{if $p\equiv1\pmod4$},\\
			\Idxstand{1+(p+1)/4} & \text{if $p\equiv3\pmod4$}.
		\end{cases}
	\]
	We can easily check that
	\[
		n-1
		=\frac{p-3}2
		\ge
		\begin{cases}
			1+\paren{p+3}/4 &\text{if $p\ge13$},\\
			1+\paren{p+1}/4 &\text{if $p\ge11$}.
		\end{cases}
	\]
	Thus, by \cref{properties of Idx}, we conclude that
	\[
		m\in\Idxstand{n-1}.
	\]
	If $e\ge2$, then $2n=\varphi(p^e)=p^e-p^{e-1}$, and in particular, $p^e\ne3^2$ since $n\ge4$.
	Then by \cref{index power of prime,properties of Idx}, we have
	\[
		m=2p^e\in\Idxstand{1+(p+e-3)}=\Idxstand{p+e-2}.
	\]
	Since $p\ge3$ and $p^e\ne3^2$, we get $p+e-2\le n-1$ by \cref{inequality}.
	Hence by \cref{properties of Idx}, we conclude that
	\[
		m\in\Idxstand{p+e-2}\subseteq\Idxstand{n-1}.
	\]
	
	Assume $\varphi\paren{m_1}=2$.
	In this case, $n=\varphi\paren{m_2}$ and $m_1\in\Idxstand1$.
	In particular, $\varphi\paren{m_2}\ge4$.
	If $\varphi\paren{m_2}=2n_2\ge6$, then the induction hypothesis implies that $m_2\in\Idxstand{n_2-1}$.
	Since $n_2=n/2\le n-1$, we conclude that
	\[
		m=m_1m_2\in\Idxstand{1+(n_2-1)}\subseteq\Idxstand{n-1}
	\]
	by \cref{properties of Idx}.
	If $\varphi\paren{m_2}=4$, then since $m_2\in\Idxstand2$, we conclude that
	\[
		m=m_1m_2\in\Idxstand{1+2}=\Idxstand{n-1}
	\]
	by \cref{properties of Idx}.
	
	Finally, assume that $\varphi\paren{m_1}=4$.
	Then $m_1\in\Idxstand2$ and $n=2\varphi\paren{m_2}$.
	If $\varphi\paren{m_2}=2n_2\ge6$, then by the induction hypothesis, we get $m_2\in\Idxstand{n_2-1}$.
	Note that $n_2+1=n/4+1\le n-1$.
	Thus, by \cref{properties of Idx},we conclude that
	\[
		m=m_1m_2\in\Idxstand{2+(n_2-1)}\subseteq\Idxstand{n-1}.
	\]
	If $\varphi\paren{m_2}=4$, then since $m_2\in\Idxstand2$, we have that
	\[
		m=m_1m_2\in\Idxstand{2+2}\subseteq\Idxstand7=\Idxstand{n-1}
	\]
	by \cref{properties of Idx}.
	If $\varphi\paren{m_2}=2$, then since $m_2\in\Idxstand1$, we conclude that
	\[
		m=m_1m_2\in\Idxstand{2+1}=\Idxstand{n-1}. \qedhere
	\]
\end{proof}

\begin{rem}
	\Cref{phi<2n and index} does not hold for $n=1,2$,
	as can be checked using \cref{index set in low dimension}.
\end{rem}

\section{Indices of smooth Calabi--Yau varieties}
\label{indices of sm cy}

In this section, we prove the first inclusion in \cref{intro main thm}, which is the main part of the theorem:

\begin{thm}
	\label{index set thm}
	Let $n\ge1$ be an integer.
	Then it holds that
	\[
		\Idx\sm n\subseteq\Idxstand{n-1}.
	\]
\end{thm}

\begin{proof}
	Show the theorem by induction on $n$.
	The case when $n=1,2$ follows from \cref{index set in low dimension}.
	
	Let $n\ge3$.
	Let $X$ be a smooth Calabi--Yau variety of dimension $n$ with index $m$.
	It suffices to show that $m\in\Idxstand{n-1}$.
	Take a Beauville--Bogomolov decomposition $f\colon\tilde X\to X$ (\cref{decomposition}).
	Take an element $g\in\Aut{\tilde X/X}$ with index $m$.
	Note that the group $\gengp g$ acts freely on $\tilde X$.
	
	First, assume that $\tilde X$ itself is strict Calabi--Yau, holomorphic symplectic, or abelian.
	Then by \cref{free aut phi<2n}, we have $\varphi\paren m\le2n$.
	Therefore, $m\in\Idxstand{n-1}$ by \cref{phi<2n and index}.
	
	We now assume that $\tilde X=\prod_i Y_i\times A$ is decomposed into at least $2$ components (of positive dimension), where the $Y_i$ are strict Calabi--Yau or holomorphic symplectic varieties, and $A$ is an abelian variety.
	We have the following two cases:
	\begin{enumerate}
		\item the automorphism $g$ is of the form $g_1\times g_2$, where $g_1\in\Aut{Y_1}$ and $g_2\in\Aut{\prod_{i\ne1}Y_i\times A}$, or
		\item otherwise, $g$ cannot be expressed in this form.
	\end{enumerate}
	
	First, assume (1), and let $Z=\prod_{i\ne1}Y_i\times A$ so that $\tilde X=Y_1\times Z$.
	Let $n_1=\dim Y_1$ and $n_2=\dim Z$, and let $m_i$ be the index of $g_i$.
	Note that $m$ divides $\lcm{m_1,m_2}$.
	
	If $g_1$ has a fixed point,
	then the action of $g_2$ on $Z$ is free since the action of $g=g_1\times g_2$ on $\tilde X$ is free.
	Therefore, $Z/\gengp{g_2}$ is a smooth Calabi--Yau variety of dimension $n_2$ with index $m_2$, so we get $m_2\in\Idx\sm{n_2}$.
	Thus, the induction hypothesis implies that $m_2\in\Idxstand{n_2-1}$.
	Moreover, by \cref{quot and index}, we see that $m_1\in\Idxstand{n_1}$ considering $Y_1/\gengp{g_1}$.
	These imply that
	\[
		m\in\Idxstand{n_1+(n_2-1)}=\Idxstand{n-1}
	\]
	by \cref{properties of Idx,divisor of index}.
	
	We now assume that $g_1$ has no fixed point.
	Then, we have $\varphi(m_1)\le2n_1$ by \cref{index of aut of no fixed points}.
	Thus, by \cref{phi<2n and index}, we see that $m_1\in\Idxstand{n_1-1}$ if $n_1\ge3$.
	This holds true for $n_1\le2$ as well.
	In fact, $m_1=1$ (resp.\;$m_1=2$) if $n_1=1$ (resp.\;$n_1=2$) by \cref{index of aut of no fixed points}.
	Moreover, we have that $m_2\in\Idxstand{n_2}$ considering $Z/\gengp{g_2}$ and by \cref{quot and index}.
	Therefore, by \cref{properties of Idx,divisor of index}, we conclude that
	\[
		m\in\Idxstand{(n_1-1)+n_2}=\Idxstand{n-1}.
	\]

	Assume (2) in the following.
	By \cref{aut of tilde X}, after renumbering if necessary, $Y_2,\dotsc,Y_e$ are isomorphic to $Y_1$ for some $2\le e\le r$, and the automorphism $g$ is of the form
	\[
		g=(h\circ\sigma)\times h'\in(\Aut{Y_1}^e\rtimes S_e)\times\AUT{\prod_{i>e}Y_i\times A}
	\]
	for some $h\in\Aut{Y_1}^e$, $\sigma\in S_e$ and $h'\in\Aut{\prod_{i>e}Y_i\times A}$.
	Since $\sigma$ can be decomposed into a product of disjoint cycles, we may choose one of these cycles and assume that $\sigma$ is a single cycle.
	Moreover, by renumbering, we assume that $\sigma$ is given by
	\[
		\sigma\colon\paren{y_1,\dotsc,y_e}
		\mapsto\paren{y_2,\dotsc,y_e,y_1}.
	\]
	We now let $Z=\prod_{i>e}Y_i\times A$, $n_1=\dim Y_1$ and $n_2=\dim Z$, so that $n=en_1+n_2$.
	Let $m_1$ be the index of $h\circ\sigma\in\Aut{Y_1^e}$, and $m_2$ be the index of $h'\in\Aut Z$.
	Note that $m$ divides $\lcm{m_1,m_2}$.
	We write $h=h_1\times\dotsb\times h_e$ with $h_i\in\Aut{Y_1}$.
	Then the powers of $h\comp\sigma$ are computed as
	\[
		\paren{h\comp\sigma}^i
		=\paren{h_1h_2\dotsm h_i\times h_2h_3\dotsm h_{i+1}\times\dotsb\times h_eh_1\dotsm h_{i+e-1}}\comp\sigma^i,
	\]
	where the $i$ in the $h_i$ denote the values in $\finset{1,\dotsc,e}$ modulo $e$.
	Thus, we observe that the automorphism $h_1\dotsm h_e$ of $Y_1$ has finite order.
	Moreover, the index of $h_1\dotsm h_e\in\Aut{Y_1}$ is $m_1$.
	Therefore, considering the quotient $Y_1/\gengp{h_1\dotsm h_e}$, we see that $m_1\in\Idxstand{n_1}$ by \cref{quot and index}.
	We also obtain that $m_2\in\Idxstand{n_2}$ considering $Z/\gengp{h'}$.
	Since $n_1+n_2<en_1+n_2=n$, it follows from \cref{properties of Idx,divisor of index} that
	\[
		m\in\Idxstand{n_1+n_2}\subseteq\Idxstand{n-1}. \qedhere
	\]
\end{proof}

\printbibliography

\end{document}